\documentclass[a4paper,11pt,reqno]{amsart}%
\usepackage{amsfonts}
\usepackage{amsthm}
\usepackage{hyperref}
\usepackage{amssymb}
\usepackage{graphicx}
\usepackage{float}
\usepackage{cite}
\usepackage{eqparbox,array}
\usepackage{eucal,times,color,enumerate,accents}
\usepackage{url}
\usepackage{bbm}
\usepackage{microtype}
\usepackage{color}
\usepackage{tikz}
\usepackage{tikz-cd}
\usetikzlibrary{arrows}
\usepackage{pifont}
\usepackage{dcolumn}
\usepackage{amsmath}
\usepackage{verbatim}%
\newcolumntype{r}{D{.}{.}{-1}}

\hypersetup{
    colorlinks,
    linkcolor={red!70!black},
    citecolor={green!60!black},
    urlcolor={blue!80!black}
}

\newtheorem{theorem}{Theorem}[section]

\theoremstyle{plain}

\newtheorem{conjecture}[theorem]{Conjecture}

\theoremstyle{definition}
\newtheorem{definition}[theorem]{Definition}
\newtheorem{example}[theorem]{Example}

\theoremstyle{plain}
\newtheorem{lemma}[theorem]{Lemma}

\newtheorem{proposition}[theorem]{Proposition}

\theoremstyle{remark}
\newtheorem{remark}[theorem]{Remark}
\numberwithin{equation}{section}

\DeclareMathOperator{\Hom}{Hom}
\DeclareMathOperator{\trop}{trop}

\DeclareMathOperator{\val}{val}

\DeclareMathOperator{\mult}{mult}
\DeclareMathOperator{\Aut}{Aut}

\newcommand {\NN}{{\mathbb N}}

\newcommand{\Q}{\mathbb{Q}}

\begin{document}

\title {Tropical curves and covers and their moduli spaces}
\author{Hannah Markwig}
\address{Hannah Markwig, Eberhard Karls Universit\"at T\"ubingen, Fachbereich Mathematik }
\email{hannah@math.uni-tuebingen.de}


\begin{abstract}
Tropical geometry can be viewed as an efficient combinatorial tool to study degenerations in algebraic geometry. Abstract tropical curves are essentially metric graphs, and covers of tropical curves maps between metric graphs satisfying certain conditions.
In this short survey, we offer an introduction to the combinatorial theory of abstract tropical curves and covers of curves, and their moduli spaces, and we showcase three results demonstrating how this theory can be applied in algebraic geometry.
\end{abstract}

\maketitle

\section{Introduction}

\label{sec-intro}

\emph{Tropical geometry} can be viewed as an efficient combinatorial tool to study degenerations in algebraic geometry. A degeneration referred to as \emph{tropicalization} associates a combinatorial object to a given algebraic variety which surprisingly captures many important properties of the algebraic variety. For that reason, tropical geometry allows to incorporate methods from discrete mathematics into algebraic geometry. Vice versa, it is also possible to use methods from algebraic geometry in discrete mathematics using tropical geometry, a prime example is the recent proof of Rota's conjecture for characteristic polynomials of matroids by Adiprasito, Huh and Katz \cite{AHK15}. 

For algebraic varieties which come embedded into projective space, tropicalization can be expressed using methods from the theory of Gr\"obner bases and initial ideals. The book \cite{MS15} by Maclagan and Sturmfels offers a broad and friendly introduction for this approach.

But also for abstract varieties, it is possible to study tropicalizations. In particular for the case of curves, abstract tropical geometry has in recent years become an attractive combinatorial theory uniting several important perspectives such as semistable reduction, Berkovich theory resp.\ non-Archimedean analytic geometry, graph theory, topology and representation theory \cite{BN09, BPR11a, ABBR13, Hel17a, Hel17}.

In this article, we offer an introduction to the \emph{combinatorial theory of abstract tropical curves and covers of curves, and their moduli spaces}, and we survey recent results demonstrating how this theory can be applied in algebraic geometry.

An (abstract) \emph{tropical curve} is, roughly, just a metrized graph, possibly with unbounded half-edges called ends. The process of tropicalization that produces a tropical curve from an algebraic curve can be pictured topologically: we think of a smooth algebraic curve as a Riemann surface and let the tubes become so thin that they end up being edges of a graph. For example, an elliptic curve, which is a torus when viewed as a Riemann surface, becomes a circle.

Moduli space of tropical curves are thus essentially spaces parametrizing certain types of graphs. It is clear that such objects show up and are important in other areas of mathematics.

In Section \ref{sec-curves}, we introduce abstract tropical curves and their moduli spaces.

Given a cover of algebraic curves, there is a way to simultaneously degenerate source and target (we can picture this topologically again, where the Riemann surface tubes become really thin) obtaining a map of graphs satisfying certain conditions --- a \emph{tropical cover}.
We discuss tropical covers and their moduli spaces in Section \ref{sec-covers}.

In Section \ref{sec-Hurwitz} and \ref{sec-Hom}, we discuss altogether three applications of the theory of tropical curves and covers and their moduli spaces in algebraic geometry.

Section \ref{sec-Hurwitz} focuses on applications in \emph{enumerative geometry}. In enumerative geometry, we take certain geometric objects (such as rational plane curves of degree $d$) and fix some conditions that they need to satisfy (e.g.\ incidence conditions: we fix points in the plane through which we require our curves to pass), and then we count how many of our geometric objects satisfy our conditions. Some of these questions go back to Ancient Greek mathematicians and are often easy to pose, but hard to answer. 

The use of tropical methods in enumerative geometry was pioneered by Mikhalkin \cite{Mi03} with his so-called \emph{correspondence theorem} proving the equality of important enumerative invariants counting numbers of plane curves satisfying incidence conditions with their tropical counterparts. Here, we discuss \emph{two variants} of so-called \emph{Hurwitz numbers}, which count covers of curves satisfying fixed ramification data (see Subsection \ref{subsec-Hurwitz}). Also for these numbers, analogous correspondence theorems hold, i.e.\  it can be shown that they are equal to their tropical counterparts. Consequently, we can use purely combinatorial methods --- merely counts of graphs satisfying certain conditions --- to determine our Hurwitz numbers, a perspective that has proved to be fruitful to derive new results about counts of covers. The correspondence theorem for our first variant (the so-called \emph{double Hurwitz numbers of the projective line}) appears in Subsection \ref{subsec-double} and for the second variant (the \emph{simple Hurwitz numbers of an elliptic curve}) in Subsection \ref{subsec-ell}.

 In Subsections \ref{subsec-piecewise} and \ref{subsec-mirror}, we showcase how correspondence theorems for Hurwitz numbers and their tropical counterparts can be applied in enumerative geometry: the first result concerns the structure of double Hurwitz numbers of the projective line viewed as a map (Subsection \ref{subsec-piecewise}), and the second concerns covers of an elliptic curve and their connection to mirror symmetry (Subsection \ref{subsec-mirror}).

There are more applications of the theory of tropical covers in enumerative geometry: another prime example is the appearance of tropical covers in the tropical computation of characteristic numbers (or Zeuthen numbers) of the plane \cite{BBM11}. The latter are well-known enumerative invariants which count plane curves satisfying incidence conditions and tangency conditions to given lines. 

Any application of tropical methods in enumerative geometry is closely related to the theory of moduli spaces: a general counting strategy in enumerative geometry is to consider a moduli space parametrizing the geometric objects we study. A condition can then be phrased as a subspace of the moduli space, and the objects satisfying all our conditions correspond to the intersection of the subspaces in question. Connections between moduli spaces in algebraic geometry and in tropical geometry continue to be intensely studied, see e.g.\ \cite{ACP12, Cap13, Ran15, Cap18, Cap18a}. One area that triggered this large attention is enumerative geometry. Also for the two variants of Hurwitz numbers which we discuss in this text, the necessary correspondence theorems --- which are at the heart of the application of tropical methods to enumerative geometry --- can be proved using the theory of moduli spaces and their tropicalizations \cite{CMR14, CMR14b}.

Tropical moduli spaces are useful not only in enumerative geometry however. In Section \ref{sec-Hom}, we show how they appear in a recent striking non-vanishing result for the cohomology of the moduli space of curves of genus $\mathcal{M}_g$ for $g\geq 2$. The cohomology of $\mathcal{M}_g$ has been a subject of intense study for decades (see e.g.\ \cite{Kir02,FPZ00,Tom05}). There have been several conjectures around which predict the vanishing of certain cohomology groups of $\mathcal{M}_g$ (among them a 25-year-old conjecture by Kontsevich) which have now been simultaneously disproved with the use of tropical methods by Chan, Galatius and Payne \cite{CGP18}.

After having introduced tropical curves and covers and their moduli spaces in Sections \ref{sec-curves} and \ref{sec-covers}, we thus offer a glimpse of the interesting applications of these objects in algebraic geometry, both in Section \ref{sec-Hurwitz} dealing with enumerative geometry and in Section \ref{sec-Hom} dealing with the cohomology of the moduli space of curves.

There is a vast literature on tropical curves, tropical covers and their moduli spaces, and this short survey cannot touch upon everything. To get a small insight into the various aspects studied in the context of tropical curves and covers, we give a small selection of recent studies in the area: Tropical curves and their Jacobians appear in \cite{MZ08}. The gonality of tropical curves is the subject of \cite{Cap12}. Tropical covers and their inverse images under tropicalization are the topic of \cite{ABBR15, ABBR13}. Tropical covers with a group action appear in \cite{LUZ19, Son19}. Combinatorial classifications of tropical covers of trees are studied in \cite{DV19}. 

We hope that this short survey serves as an entry point into the subject and guides the readers towards others of the many interesting facets of the theory of tropical curves and covers and their moduli spaces.

\subsection{Acknowledgements} We thank Sam Payne and JuAe Song for helpful comments.

\section{Tropical curves}\label{sec-curves}
An \emph{(abstract) tropical curve} is a connected metric graph $\Gamma$ with unbounded rays called \emph{ends} and a genus function $g:\Gamma\rightarrow \NN$ which is nonzero only at finitely many points. Locally around a point $p$, $\Gamma$ is homeomorphic to a star with $r$ half-edges. 
The number $r$ is called the \emph{valence} of the point $p$ and denoted by $\val(p)$.
We require that there are only finitely many points with $\val(p)\neq 2$. 
A finite set of points containing  (but not necessarily equal to the set of) all points of nonzero genus or valence larger than $2$ may be chosen; its elements  are called  \emph{vertices}.  (If we do not specify differently, we choose as vertices the set of all  points of nonzero genus or valence larger than $2$.)
By abuse of notation, the underlying graph with this vertex set is also denoted by $\Gamma$.
Correspondingly, we can speak about \emph{edges} and \emph{flags} of $\Gamma$. A flag is a tuple $(v,e)$ of a vertex $v$ and an edge $e$ with $v\in \partial e$. It can be thought of as an element in the tangent space of $\Gamma$ at $v$, i.e.\ as a germ of an edge leaving $v$, or as a half-edge (the half of $e$ that is attached to $v$). Edges which are not ends have a finite length and are called \emph{bounded edges}.

A \emph{marked tropical curve} is a tropical curve whose ends are (partially) labeled. An isomorphism of a tropical curve is an isometry respecting the end markings and the genus function. The \emph{genus} of a tropical curve $\Gamma$ is given by
\[
g(\Gamma) = h_1(\Gamma)+\sum_{p\in \Gamma} g(p).
\]
A curve of genus $0$ is called \emph{rational}. A curve satisfying $g(v)=0$ for all $v$ is called \emph{explicit}. The \emph{combinatorial type} is the equivalence class of tropical curves obtained by identifying any two tropical curves which differ only by edge lengths.

\begin{example}[All combinatorial types of rational tropical curves with $4$ marked ends]\label{ex-m041}
Figure \ref{fig-m04} shows all combinatorial types of rational tropical curves with $4$ marked ends which are at least $3$-valent. We do not depict the genus function if the curve is explicit. To pick a tropical curve of one of the three left combinatorial types, we have to fix a length $l$ for the bounded edge connecting the two pairs of ends, e.g. $l=1$.
\begin{figure}
\begin{center}
\input{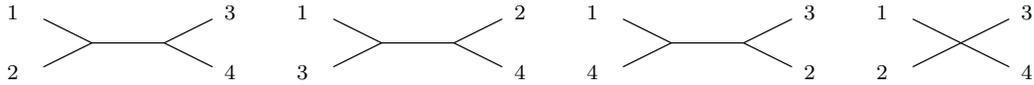}
\end{center}\caption{All at least $3$-valent combinatorial types of rational $4$-marked tropical curves.}\label{fig-m04}
\end{figure}
\end{example}

\begin{example}\label{ex-m12}
Figure \ref{fig-m12} shows all combinatorial types of tropical curves of genus $1$ with $2$ marked ends which are at least $3$-valent. We depict a vertex of genus one as a black dot. We do not label the ends since they are not distinguishable in every type, so there is a unique way to add the numbers $1$  and $2$ to the ends for each picture.
\begin{figure}
\begin{center}
\input{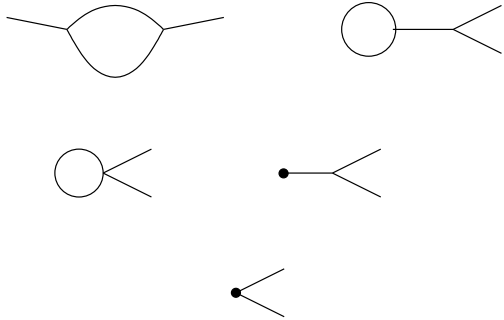}
\end{center}\caption{All combinatorial types of $2$-marked tropical curves of genus $1$.}\label{fig-m12}
\end{figure}
\end{example}

\begin{example}[Two combinatorial types of tropical curves of genus $2$]\label{ex-m2}
Figure \ref{fig-m2} shows two combinatorial types of tropical curves of genus $2$ with no marked ends.
\begin{figure}
\begin{center}
\input{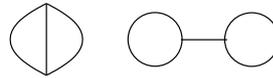}
\end{center}\caption{Two combinatorial types of  tropical curves of genus $2$ with no marked ends.}\label{fig-m2}
\end{figure}
\end{example}

\subsection{Moduli spaces of tropical curves}\label{subsec-moduliofcurves}
The set of all tropical curves of a fixed combinatorial type is a positive orthant, specifying (positive) lengths for each bounded edge. If the underlying graph has automorphisms, then we must identify points of the orthant accordingly.
\begin{example}[(Possibly folded) orthants parametrize tropical curves of a given combinatorial type]\label{ex-folding}
Consider the top left combinatorial type of $2$-marked tropical curve of genus $1$ in Figure \ref{fig-m12}. Denote the lengths of the bounded edges with $a$ and $b$. Then the tuple $(a,b)\in(\mathbb{R}_{>0})^2$ has to be identified with $(b,a)$, because of the automorphism of the underlying graph. The set of all tropical curves of this combinatorial type can thus be viewed as a "folded" positive orthant, see Figure \ref{fig-folded}.
\begin{figure}
\begin{center}
\input{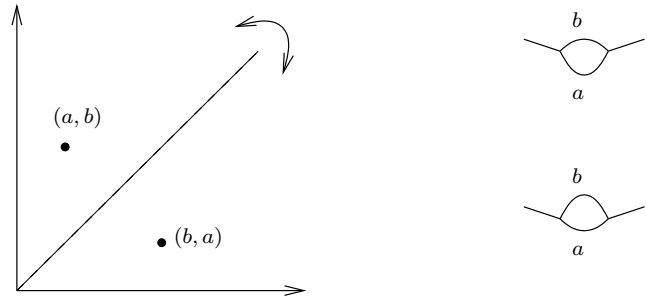}
\end{center}\caption{A folded positive orthant parametrizes all tropical curves of the top left combinatorial type in Figure \ref{fig-m12}. The point $(a,b)$ corresponds to the curve on the upper right, the point $(b,a)$ to the curve on the lower right, which are identified by the automorphism of the underlying graph, switching the two parallel edges.}\label{fig-folded}
\end{figure}
\end{example}
\begin{definition} The \emph{dimension} of a combinatorial type is the number of bounded edges.
\end{definition}
By the above, this equals the dimension of the (possibly folded) positive orthant that parametrizes all curves of the given combinatorial type.

We can see that (possibly folded) positive orthants  are the building blocks that we need if we want to work with a space parametrizing all tropical curves of a given genus and with a fixed number of marked ends. Orthants corresponding to different combinatorial types $\alpha$ and $\beta$ have to be glued if $\alpha$ is obtained from $\beta$ by shrinking edges to length $0$ and identifying their adjacent vertices. More precisely, the orthant corresponding to $\alpha$ can be identified with a coordinate subspace at the boundary of the orthant for $\beta$, where the coordinates which are set zero correspond precisely to the shrinking edges.
An object which is, roughly, glued from cones in such a way (where self-maps among cones as the foldings we discussed in Example \ref{ex-folding} are allowed) is called an \emph{abstract cone complex} \cite{ACP12}.

\begin{definition}[Moduli spaces of tropical curves]
The moduli space $M_{g,n}^{\trop}$ is the abstract cone complex parametrizing all tropical curves of genus $g$ with $n$ marked ends which are at least $3$-valent. If $n=0$, we write $M_{g}^{\trop}$. 
\end{definition} 

\begin{example}[$M_{1,2}^{\trop}$ as abstract cone complex]\label{ex-mgtrop}
Consider the space $M_{1,2}^{\trop}$. The combinatorial types of tropical curves of genus $1$ with two marked ends are depicted in Figure \ref{fig-m12}, where they have been arranged in rows according to their dimension. The two top combinatorial types produce two $2$-dimensional positive orthants, of which one is folded as discussed in Example \ref{ex-folding}. The left type in the second row corresponds to the open ray which is glued between the two $2$-dimensional cones. The right type in the second row is the open ray bounding the other side of the non-folded $2$-dimensional orthant. The combinatorial type in the last row finally gives the vertex. Figure \ref{fig-m12glued} depicts the abstract cone complex $M^{\trop}_{1,2}$. Drawn in red is the simplicial complex (with self-gluing resp.\ folding as induced by the folded orthant) that is given by the subset of all tropical curves whose edge lengths sum up to $1$. 
\begin{figure}
\begin{center}
\input{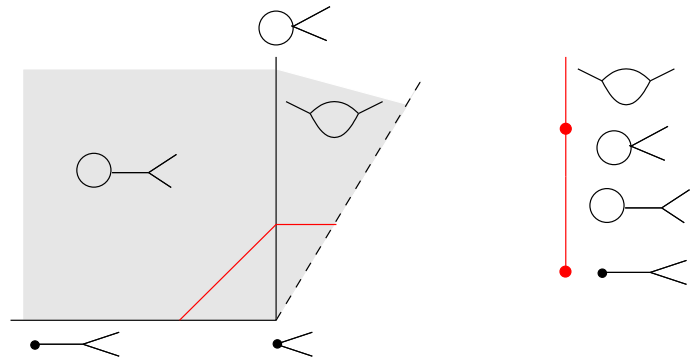}
\end{center}\caption{The cone complex $M_{1,2}^{\trop}$. On the right, the simplicial complex obtained by restricting to curves whose total edge lengths sum up to $1$.}\label{fig-m12glued}
\end{figure}
\end{example}

\begin{example}[$M_2^{\trop}$ and its universal unfolding]
To discuss $M_2^{\trop}$, we restrict to the subset of tropical curves whose edge lengths sum up to $1$, obtaining a simplicial complex as in Example \ref{ex-mgtrop}. Notice that the two combinatorial types of genus $2$ curves depicted in Figure \ref{fig-m2} are the two top-dimensional types. They both yield $3$-dimensional orthants, resp.\ $2$ dimensional simplices. The simplex corresponding to the left picture comes with an action induced by the automorphism group of the underlying graph, which is an $\mathbb{S}_3$-action. It is thus folded along the three middle lines. The combinatorial type on the right yields a simplex which is folded only once. Readers familiar with Outer space \cite{CM86} may notice that the universal unfolding of the explicit part of this subset of $M_2^{\trop}$ equals Outer space in rank $2$ (see Figure 2 in \cite{Vog08}).
\end{example}

\begin{example}[$M_{0,4}^{\trop}$ as cone complex]
The three left combinatorial types in Figure \ref{fig-m04} give three $1$-dimensional positive orthants, i.e.\ open half rays, which are glued at a vertex that corresponds to the right combinatorial type. The space $M_{0,4}^{\trop}$ thus looks like a rational tropical curve with $3$ ends \cite{Mi06}.
\end{example}

\section{Tropical covers}\label{sec-covers}
\begin{definition} [Tropical covers]
\label{def:tropcov}
A \emph{tropical cover} $\pi:\Gamma_1\rightarrow \Gamma_2$ is a surjective map of abstract tropical curves satisfying:
\begin{enumerate} 
\item \emph{Locally integer affine linear.} The map $\pi$ is piecewise integer affine linear, the slope of $\pi$ on a flag or edge $e$ is a positive integer called the \emph{weight} (or expansion factor) $\omega(e)\in \NN_{> 0}$. 
\item \emph{Harmonicity/Balancing.} For a point $v\in \Gamma_1$, the \emph{local degree of $\pi$ at $v$} is defined as follows. Choose a flag $f'$ adjacent to $\pi(v)$, and add the weights of all flags $f$ adjacent to $v$ that map to $f'$:
\begin{equation}
d_v=\sum_{f\mapsto f'} \omega(f).
\end{equation} 
We require that $\pi$ is \emph{harmonic} (resp.\ \emph{requires the balancing condition}) at every vertex, i.e.\ that for each $v\in \Gamma_1$, the local degree at $v$ is well defined (i.e. independent of the choice of $f'$) \cite{BN09}.
\end{enumerate}
\end{definition}
We choose the vertex sets for the source and the target of a tropical cover in the minimal way such that it is a map of graphs, i.e.\ we declare images and preimages of vertices to be vertices.

A tropical cover is called a \emph{tropical Hurwitz cover} if it satisfies the \emph{local Riemann-Hurwitz condition} at each point $v\in \Gamma_1$ \cite{BBM10, Cap12, CMR14} stating that when $v\mapsto v'$ with local degree $d_v$,
\begin{equation}\label{eq-RH}
0\leq d_v(2-2g(v'))-\sum (\omega(e)-1) - (2-2g(v)),
\end{equation}
 where the sum goes over all flags $e$ adjacent to $v$. The number on the right hand side of the inequality can be viewed as a measure for the ramification at $v$, which cannot be less than zero.

The \emph{degree} of a tropical cover is the sum over all local degrees of preimages of a point $a$, $d=\sum_{p\mapsto a} d_p$. By the balancing condition, this definition does not depend on the choice of $a\in \Gamma_2$. For an end $e\in \Gamma_2$ let $\mu_e\vdash d$ be the partition of weights of the ends of $\Gamma_1$ mapping onto $e$. We call $\mu_e$ the \emph{ramification profile} above $e$. 

The \emph{combinatorial type} of a tropical cover is the data obtained when dropping the metric information, i.e.\ it consists of the combinatorial types of the abstract tropical curves together with the weights and the data which edge maps to which.

\begin{example}
Figure \ref{fig-cover} depicts a (combinatorial type of a) tropical cover. The map of abstract tropical curves is harmonic resp.\ balanced, since it has degree $2$ everywhere: above points in the left ends, it has two preimages of weight $1$ each, above points in the bounded edge and in the right ends, it has one preimage of weight $2$. The left vertex satisfies the local Riemann-Hurwitz condition, since $0\leq 2\cdot2 -1-2$. The right vertex does not satisfy it, since $0> 2\cdot2 -1-1-1-2$. 
The ramification profile over the left ends is $(1,1)$, over the right ends it is $(2)$.
\begin{figure}
\begin{center}
\input{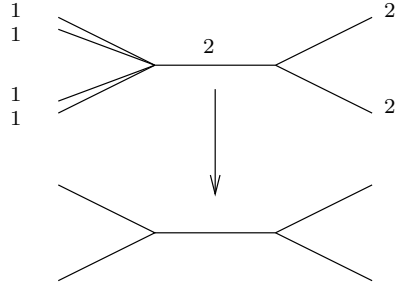}
\end{center}\caption{A tropical rational curve covering another. The numbers at the edges are the weights. We do not specify lengths of edges, the lengths in the source and in the target are related via the weight. (Put differently, we depict only a combinatorial type.) The map is a tropical cover, since it satisfies the balancing condition. It is not a tropical Hurwitz cover, since the local Riemann-Hurwitz condition is violated at the right vertex.}\label{fig-cover}
\end{figure}
\end{example}
Notice that the local Riemann-Hurwitz condition is a \emph{realizability} condition: only tropical covers satisfying it at all points can be degenerations of covers of algebraic curves.

\begin{remark}\label{rem-lengthsrelated}
Given a tropical cover $\pi:\Gamma_1\rightarrow \Gamma_2$, the combinatorial type together with the lengths of $\Gamma_2$ determine the metric of $\Gamma_1$: this is true because for each edge $e_1$ of weight $\omega(e_1)$ and length $l(e_1)$, the length of its image satisfies $l(e_2)=\omega(e_1)\cdot l(e_1)$ by Condition (1) of Definition \ref{def:tropcov}. For that reason, we often do not specify length data in pictures.
\end{remark}

\subsection{Moduli spaces of tropical covers}
It follows from Remark \ref{rem-lengthsrelated} that for a fixed combinatorial type of tropical cover, the orthant parametrizing all possible metrizations for the target curve also parametrizes all tropical covers of this type.
As in Subsection \ref{subsec-moduliofcurves}, we can glue such (possibly folded) orthants to obtain a moduli space of tropical covers. We usually fix the genus, the number of ends of the target curve, the degree, all ramification profiles above ends, and the genus of the source curve.

\begin{example}[A moduli space of tropical covers]\label{ex-tropdHcov}
Let the target curve be of genus $0$ and have two ends. We consider covers with a source curve of genus $1$ and with two ends, of degree $3$. Each end must be covered by ends. We require the ramification profile above each end to be $(3)$. 
Then there is only one top-dimensional combinatorial type of tropical cover of this form, as depicted in Figure \ref{fig-cover2}.
\begin{figure}
\begin{center}
\input{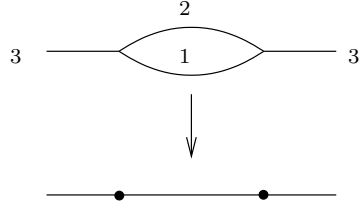}
\end{center}\caption{The top-dimensional combinatorial type of tropical cover for a genus $1$ curve with two ends of ramification profiles $(3)$ covering a genus $0$ curve with $2$ ends with degree $3$.}\label{fig-cover2}
\end{figure}
It is possible to generalize the definition of tropical cover in a way that allows to shrink edges to a point, then also the top right type in Figure \ref{fig-m12} could be the source of a cover.
Without allowing edges of weight $0$, the moduli space of covers of this form is a single ray, corresponding to the combinatorial type depicted in Figure \ref{fig-cover2}. We can apply the forgetful map forgetting the covering map, which takes this ray to a ray of slope $2$ in the right cone of $M^{\trop}_{1,2}$ in Figure \ref{fig-m12glued}.

\end{example}

\section{Counting covers}\label{sec-Hurwitz}

\subsection{Counting covers of Riemann surfaces}\label{subsec-Hurwitz}
A question going back to Hurwitz asks for the number of Riemann surfaces of a fixed genus $g$ covering another Riemann surface of genus $h$ with fixed ramification data. These so-called \emph{Hurwitz numbers} continue to play an important role as a theme connecting various areas of mathematics such as combinatorics, representation theory, geometry and mathematical physics \cite{CM16}. 

A cover of Riemann surfaces is branched over finitely many points. 
Locally around a point $p$ of the source curve, the cover is given by a map of the form $z\mapsto z^r$. The number $r$ is called the ramification index of $p$.
The \emph{ramification profile} of a point in the target is the partition encoding the ramification indices of its preimages. Every point which is not a branch point has ramification profile $(1,\ldots,1)$ . If a point has ramification profile $(2,1,\ldots,1)$, it is called a \emph{simple branch point}.

In this note, we focus on two particular cases.
\begin{enumerate}
\item We consider \emph{double Hurwitz numbers} $H_g(\mu,\nu)$ which count covers of the projective line with ramification profiles $\mu$ (resp.\ $\nu$) over $0$ (resp.\ $\infty$) and only simple branch points else, at fixed points. The number $s$ of simple branch points to be fixed is given by the Riemann-Hurwitz formula
$$s=  -2+2g+\ell(\mu)+\ell(\nu),$$ 
where $\ell(\mu)$ stands for the number of parts of $\mu$.
\item We consider \emph{simple Hurwitz numbers of an elliptic curve} $N_{d,g}$ which count degree $d$, genus $g$  covers of an elliptic curve $E$ with $s$ simple branch points at fixed points, where now $s=2g-2$.
\end{enumerate}
In both counts, each cover is weighted by one over the number of its automorphisms.

In the following, we show how these two versions of Hurwitz numbers can be obtained tropically and why the tropical approach is useful.

\subsection{Tropical double Hurwitz numbers}\label{subsec-double}
In the tropical count corresponding to double Hurwitz numbers, we let the target $\Gamma_2$ be a straight line and require ramification profiles $\mu$ and $\nu$ above the two ends.
We insert $s$ $2$-valent vertices of genus zero at arbitrary but fixed points, and we require these to be images of ramification points, i.e.\ of points for which the local Riemann Hurwitz condition (\ref{eq-RH}) is a strict inequality.

\begin{definition}[Tropical double Hurwitz covers] \label{def-tropdHcov}
Fix integers $s$, $g$ and partitions $\mu$ and $\nu$ of the same number, satisfying $s=  -2+2g+\ell(\mu)+\ell(\nu)$.

 Let $\Gamma_2$ be a straight line with $s$ vertices of genus $0$.
 A \emph{tropical double Hurwitz cover} (for the fixed data) is a tropical Hurwitz cover $\pi:\Gamma_1\rightarrow\Gamma_2$, where
 \begin{enumerate}
 \item $\Gamma_1$ is of genus $g$,
 \item the ramification profile over the left end of $\Gamma_1$ is $\mu$ and over the right end $\nu$,
 \item the preimage of each of the $s$ vertices contains a point for which the local Riemann Hurwitz condition (\ref{eq-RH}) is a strict inequality.
 \end{enumerate}

\end{definition}

\begin{example} The cover depicted in Figure \ref{fig-cover2} is a tropical double Hurwitz cover where $s=2$, $g=1$ and $\mu=\nu=(3)$. 
\end{example}
 
\begin{lemma}\label{lem-tropdoubH} Let $\pi:\Gamma_1\rightarrow \Gamma_2$ be a tropical double Hurwitz cover. Then $\Gamma_1$ is explicit and has $s$ $3$-valent vertices (and only $2$-valent vertices else), one in each preimage of one of the $s$ vertices of $\Gamma_2$.
\end{lemma}
\begin{proof}

Notice that the local Riemann Hurwitz condition (\ref{eq-RH}) for a cover of a straight line simplifies to $0\leq \val(v)-2+2g(v)$. This is true since the sum over all flags adjacent to $v$ can be arranged into two sums, one for each side of the image vertex to which a flag can map. The sum of the weights of the flags mapping to the left side equals $d_v$, and the same holds for the sum of the weights of the flags mapping to the right. Thus, a contribution of $2d_v$ cancels away in the first two summands of the right hand side of (\ref{eq-RH}) and we are left with the valency of $v$.
We can have a strict inequality only if $\val(v)\geq 3$ or $g(v)\geq 1$. Let the set of vertices of $\Gamma_1$ momentarily be given by all points of genus bigger $0$ or valence bigger $2$, and denote the number of these vertices by $V$. By the above, we can conclude $$V\geq s,$$ since there must be a point with a strict inequality mapping to each of the $s$ vertices of $\Gamma_2$.

By our requirement, the source $\Gamma_1$ has genus $g$ and $\ell(\mu)+\ell(\nu)$ ends. Let $g'$ be the genus of the underlying graph of $\Gamma_1$, i.e.\ $g'=g-\sum_pg(p)$. The number of bounded edges is called $E$. Inductively, we can see that $$E\leq \ell(\mu)+\ell(\nu)+3g'-3,$$ where equality is attained if and only if each vertex is $3$-valent. On the other hand, the Euler characteristic of $\Gamma_1$ is $V-E=1-g'$, so we have \begin{align*}&V-(\ell(\mu)+\ell(\nu)+3g'-3)\leq 1-g' \Rightarrow  \\ & V\leq -2+2g'+\ell(\mu)+\ell(\nu)\leq  -2+2g+\ell(\mu)+\ell(\nu)=s.\end{align*} Combining with the above, we can see $g=g'$, and $E=  \ell(\mu)+\ell(\nu)+3g-3$, i.e.\ each vertex is $3$-valent and of genus zero.
\end{proof}

\begin{remark}\label{rem-interiorpt}
It follows from Lemma \ref{lem-tropdoubH} that for fixed $\Gamma_2$, each tropical double Hurwitz cover corresponds to a unique point in the interior of a top-dimensional cell of the moduli space of tropical genus $g$ covers of  a straight line with $s$ vertices and ramification profiles $\mu$ and $\nu$ over the two ends.
\end{remark}

\begin{remark}[Two views on multiplicities]\label{rem-mult}
In tropical geometry, it is common to count objects with \emph{multiplicity}.
This multiplicity can be thought of in two natural ways:
\begin{itemize}
\item We can view tropical geometry as a degenerate version of algebraic geometry. There are many ways to make the degeneration process precise, all of which are usually referred to as \emph{tropicalization}. For the purpose of this article, we can think of a merely topological degeneration taking a Riemann surface with very thin tubes to a graph.

If we want to use tropical methods for counting objects in algebraic geometry, it is then obvious that we should count with a multiplicity that takes into account how many of our algebraic objects tropicalize to a given tropical one. More concretely, the multiplicity of a tropical double Hurwitz cover  $\pi:\Gamma_1\rightarrow \Gamma_2$  equals the number of covers of Riemann surfaces contributing to $H_g(\mu,\nu)$ which degenerate to $\pi$. 
\item Consider the moduli space of all genus $g$ tropical covers of a straight line with $s$ vertices and ramification profiles $\mu$ and $\nu$ over the two ends. By Remark \ref{rem-interiorpt}, each tropical double Hurwitz cover corresponds to a unique point in the interior of a top-dimensional cell of this moduli space. More precisely, the unique point in a cell is cut out by linear equations fixing the positions of the $s$ vertices. The multiplicity can also be viewed as the (tropical) intersection multiplicity of these linear equations with the cell.
\end{itemize}
It is shown in \cite{CJM10} that both perspectives yield the same multiplicity, and, moreover, that this multiplicity can be given combinatorially in terms of weights of edges as we define now.
\end{remark} 

\begin{definition}[Multiplicity of a tropical double Hurwitz cover] Let $\pi:\Gamma_1\rightarrow \Gamma_2$ be a tropical double Hurwitz cover. We define its multiplicity, $\mult(\pi)$, to be the product of the weights of its bounded edges divided by the number of automorphisms.
\end{definition}
It is shown in \cite{CJM10} that automorphisms of a tropical double Hurwitz cover can only arise because of
\begin{itemize} \item ends of the same weight adjacent to the same vertex ("balanced forks") or
\item cycles formed by two edges of the same weight ("balanced wieners").
\end{itemize}
 Thus
 \begin{equation}
 \mult(\pi)=\left(\prod_e \omega(e) \right)\frac{1}{2^{f+w}},
 \label{eq-mult}
 \end{equation}
 where the product goes over all bounded edges $e$, $f$ denotes the number of balanced forks and $w$ the number of balanced wieners.
 \begin{example}\label{ex-mult}
 The tropical double Hurwitz cover in Figure \ref{fig-cover2} has multiplicity $2\cdot 1=2$.
 \end{example}

 \begin{definition} [Tropical double Hurwitz number]
 Fix positive integers $s$, $g$ and partitions $\mu$ and $\nu$ of the same number, satisfying $s=  -2+2g+\ell(\mu)+\ell(\nu)$.
 Fix a target $\Gamma_2$ as in Definition \ref{def-tropdHcov}.
 The \emph{ tropical double Hurwitz number} $H_g^{\trop}(\mu,\nu)$ is the weighted count of tropical double Hurwitz covers, where each cover is counted with its multiplicity as given in (\ref{eq-mult}).
 \end{definition}

\begin{example}
Let $s=2$, $g=1$ and $\mu=\nu=(3)$. We have seen in Example \ref{ex-tropdHcov} that the cover depicted in Figure \ref{fig-cover2}  is the only tropical double Hurwitz cover to be counted in this situation, and in Example \ref{ex-mult} we have seen that its multiplicity is $2$. Thus $H_1^{\trop}((3),(3))=2$.
     \end{example}

\begin{theorem} [The correspondence theorem for double Hurwitz numbers] \label{thm-corres} Fix an integer $g$ and partitions $\mu$ and $\nu$ of the same number. Then the double Hurwitz number equals its tropical counterpart, i.e.\
$$ H_g(\mu,\nu)=H_g^{\trop}(\mu,\nu).$$
\end{theorem}
If we take the first perspective of Remark \ref{rem-mult}, the equality in the correspondence theorem is obvious. But we have defined the multiplicity combinatorially as a product of edge weights, and so the equality is actually a theorem that has to be proved. There are several ways to prove it: in \cite{CJM10}, we use monodromy representations in the symmetric group, in \cite{BBM10} we use topological degenerations, in \cite{CMR14} we use moduli spaces and intersection theory.
The idea for the latter is to show
\begin{enumerate}
\item that the moduli space of tropical covers is the tropicalization of a suitable corresponding moduli space in algebraic geometry, and
\item that the intersection corresponding to the fixing of branch points in algebraic geometry tropicalizes to the tropical intersection which then yields the multiplicity as mentioned in the second item of Remark \ref{rem-mult}. 
\end{enumerate}
Such a strategy can be used in a much broader context of course and yields correspondence theorems for a range of interesting numbers in enumerative geometry, see e.g.\ \cite{Ran15, ACP12}.

\subsection{Piecewise polynomiality of double Hurwitz numbers}\label{subsec-piecewise}

The main result in the previous section is the correspondence theorem stating the equality $H_g(\mu,\nu)=H_g^{\trop}(\mu,\nu)$. With this equality, we obtain an alternative, purely combinatorial approach to Hurwitz numbers: we can compute double Hurwitz numbers by listing certain graphs and adding up their (combinatorially defined) multiplicities. 
This approach is e.g.\ useful for structural results, as we now describe.

Let $H=\{(\mu,\nu)\in \NN^{\ell(\mu)+\ell(\nu)};\; \sum_i \mu_i=\sum_j \nu_j\}$ be the set of pairs of partitions of an equal number. We can view double Hurwitz numbers as a function $$H_g:H\rightarrow \mathbb{Q}$$ sending a tuple $(\mu,\nu)$ to $H_g(\mu,\nu)$. 
It is interesting to study this function, because it simultaneously describes all double Hurwitz numbers of genus $g$. It turns out that it has intriguing structural results: it is piecewise polynomial in the entries $\mu_i$, $\nu_j$ \cite{gjv:ttgodhn}. We can also describe the leading and the smallest degree of the polynomials and the positions of the walls bounding the loci of polynomiality. Furthermore, we can give wall-crossing formulas, i.e.\ differences among polynomials for neighbouring loci of polynomiality, which can interestingly be given in terms of Hurwitz numbers with smaller input data \cite{ssv:cbodhn}.
Part of this nice structural behaviour of the double Hurwitz number function was known for two decades, but for $g>0$ the positions of the walls and wall-crossing formulas were first discovered with the help of tropical geometry in 2010 \cite{cjm:wcfdhn, Joh10}.

It turns out that all of the structural results appear rather naturally on the tropical side; we would like to demonstrate this here by proving the piecewise polynomiality in genus $0$, i.e.\ of $H_0(\mu,\nu)$.
Furthermore, the methods find applications for other variants of the Hurwitz problem, e.g.\ the piecewise polynomiality of monotone Hurwitz numbers \cite{Hah19, HL19}.

\begin{theorem}[Piecewise polynomiality of double Hurwitz numbers]
The genus $0$ double Hurwitz number function $$H_0:H\rightarrow \mathbb{Q}:(\mu,\nu)\mapsto H_0(\mu,\nu)$$ is piecewise polynomial in the entries $\mu_i,\nu_j$. The walls separating the regions of polynomiality are given by expressions of the form 
$$\sum_{i\in I}\mu_i - \sum_{j\in J}\nu_j=0, $$
where $I$ is a subset of $\{1,\ldots,\ell(\mu)\}$ and $J$ of $\{1,\ldots,\ell(\nu)\}$.
\end{theorem}
\begin{proof}
We build on the correspondence Theorem \ref{thm-corres} stating $H_0(\mu,\nu)=H_0^{\trop}(\mu,\nu)$ and investigate the piecewise polynomiality of $H_0^{\trop}(\mu,\nu)$.

Start by listing all $3$-valent rational graphs with $n=\ell(\mu)+\ell(\nu)$ ends. For each edge, pick an orientation reflecting the way it should be mapped to the target straight line, i.e.\ the source of the edge shall be mapped to the left of the target of the edge. Accordingly, the ends with weights $\mu_i$ are always oriented towards their unique vertex, and the ends with weights $\nu_j$ out of their unique vertex.

These oriented graphs appear as candidates for the source of tropical double Hurwitz covers. Since the graph is rational, the balancing condition uniquely specifies the weight of each edge in terms of the entries $\mu_i$ and $\nu_j$. More precisely, the edge separating the ends with weights $\mu_i$, $i\in I$ and $\nu_j$, $j\in J$ for subsets $I$ and $J$ of $\{1,\ldots,\ell(\mu)\}$ and $\{1,\ldots,\ell(\nu)\}$, respectively, of the ends with weights $\mu_i$, $i\notin I$ and $\nu_j$, $j\notin J$ must have weight $\pm \sum_{i\in I}\mu_i - \sum_{j\in J}\nu_j$ (depending on the chosen orientation). 
A graph qualifies as the source of a tropical double Hurwitz cover if and only if the weights for its edges are positive, i.e.\ if and only if the expressions $\pm (\sum_{i\in I}\mu_i - \sum_{j\in J}\nu_j)>0$ for each edge, splitting the ends in terms of $I$ and $J$ as discussed above. (This implies that for each unoriented graph, one choice of orientation qualifies for each region.)
It follows that the walls are given as above, and that a region of polynomiality is defined by fixing a list of inequalities of the form  $\pm (\sum_{i\in I}\mu_i - \sum_{j\in J}\nu_j)>0$. 

For a fixed region, we then have a finite list of oriented graphs as above which qualify. A given oriented graph may admit more than one tropical double Hurwitz cover, this is the case if there are several ways to map the vertices to their $s$ image points in the straight line $\Gamma_2$ which respects the conditions on the map imposed by the orientation. (For example, if, as in Figure \ref{fig-cycleorder}, there is a cycle formed by four edges of which two connect the vertices $V_1$ and $V_4$ via $V_2$ (positively oriented), and two via $V_3$, then the image of $V_2$ may be smaller or larger than the image of $V_3$.)

For a fixed choice of images of the $s$ vertices, we then obtain a unique tropical double Hurwitz cover, which has to be counted with the product of their edge weights, i.e.\ with a product of expressions of the form $\pm (\sum_{i\in I}\mu_i - \sum_{j\in J}\nu_j)$ which is polynomial in the $\mu_i$ and $\nu_j$. 
It follows that the overall count is polynomial in the $\mu_i$ and $\nu_j$.
\end{proof}

\begin{figure}
\begin{center}
\input{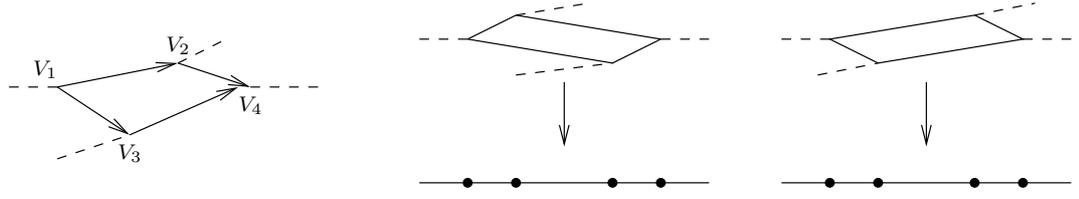}
\end{center}\caption{An oriented source graph may admit several tropical double Hurwitz covers, by picking images for vertices. Notice that the orientation implies that $V_1$ must be mapped before $V_2$ and $V_4$, and $V_1$ before $V_3$ before $V_4$, but the order of the images of $V_2$ and $V_3$ can be chosen.}\label{fig-cycleorder}
\end{figure}

\begin{example}\label{ex-walls}
Fix $\ell(\mu)=\ell(\nu)=2$.
Figure \ref{4ends} shows oriented rational graphs with $2$ in-ends and $2$ out-ends, and with the suitable weight for the bounded edge, satisfying the balancing condition.  
\begin{figure}[tb]
\begin{center}
\input{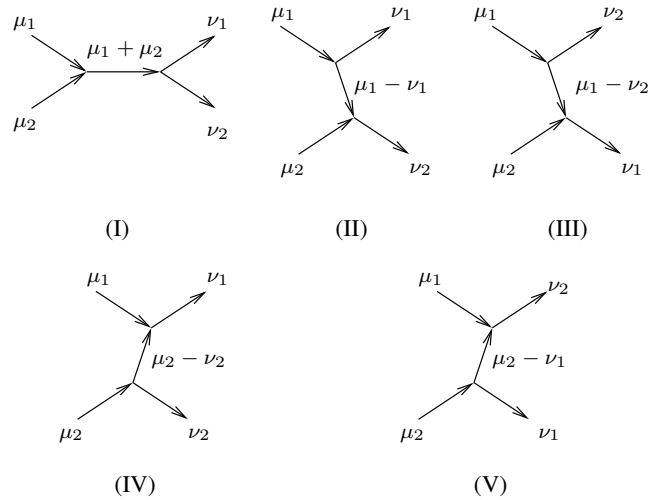}
\end{center}
\caption{Rational graphs with two in-ends and two out-ends, and weights satisfying balancing.}
\label{4ends}
\end{figure}

The regions of polynomiality are defined by the inequalities
$\pm(\mu_1- \nu_1) {>}0$, and $\pm(\mu_1-\nu_2){>} 0$.
Note that two such inequalities, e.g. $\mu_1 > \nu_1$ and $\mu_1 > \nu_2$ imply two other inequalities $\nu_1 >\mu_2$ and $\nu_2 >\mu_2$. This is true since the sum $\mu_1+\mu_2$ equals $\nu_1+\nu_2$.
Figure \ref{walls-ex} shows the four regions and the two walls, and marks which of the above graphs belongs to  each chamber. Also, it shows the polynomial  which equals $H_0$ in each chamber.
\begin{figure}[htb]
\begin{center}
\input{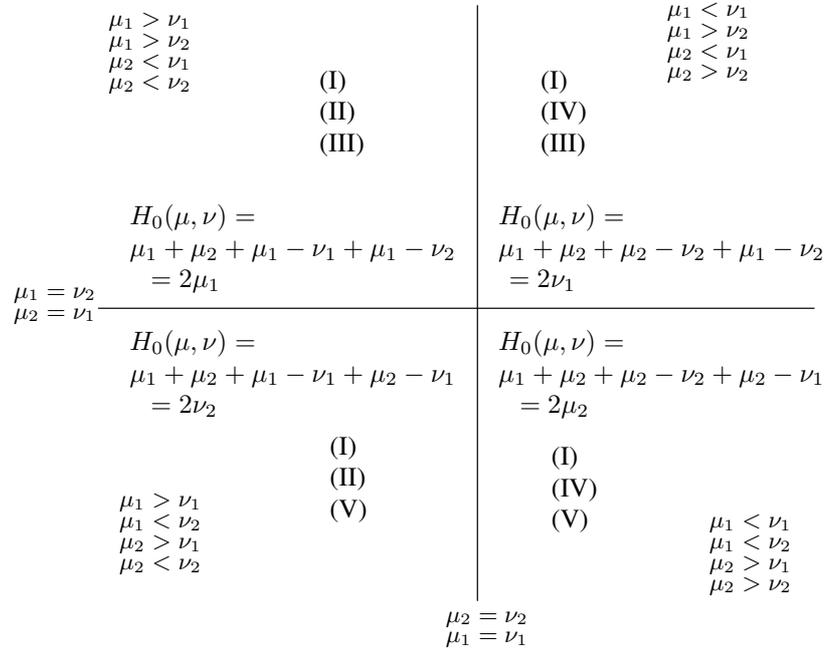}
\end{center}
\caption{The polynomiality regions for $H_0((\mu_1,\mu_2),(\nu_1,\nu_2))$.}
\label{walls-ex}
\end{figure}
\end{example}

\subsection{Tropical simple Hurwitz numbers of an elliptic curve}
\label{subsec-ell}

Also the second counting problem described in Section \ref{subsec-Hurwitz} can be obtained tropically: the count of simply ramified covers of an elliptic curve.

Viewed as Riemann surface, an elliptic curve is a torus. In light of the degeneration process described in Remark \ref{rem-mult} (1), a tropical elliptic curve is thus a circle, equipped with a length. We fix $g\geq 2$ and $2g-2$ points on the circle as vertices, above which we require a preimage for which the local Riemann-Hurwitz condition is a strict inequality. We call this target tropical elliptic curve with $2g-2$ vertices $E$. Analogously to Lemma \ref{lem-tropdoubH}, we can then prove that the source of each tropical Hurwitz cover of $E$ is explicit, and has precisely one $3$-valent vertex above each of the $2g-2$ vertices.

\begin{example}
Figure \ref{fig-tropcover} shows a tropical Hurwitz cover $\pi:\Gamma\rightarrow E$ of degree $4$ with a genus $2$
source curve. The red numbers close to the vertex $P$ are the weights of the corresponding edges, the black numbers denote the lengths. The cover is balanced at
$P$ since there is an edge of weight $3$ leaving in one direction and an
edge of
weight $2$ plus an edge of weight $1$ leaving in the opposite direction.

As discussed in Remark \ref{rem-lengthsrelated}, the length of an edge of $\Gamma$ is determined by its weight and
the length of its image, and we will therefore not specify edge lengths in the next pictures.

\begin{figure}
 \input{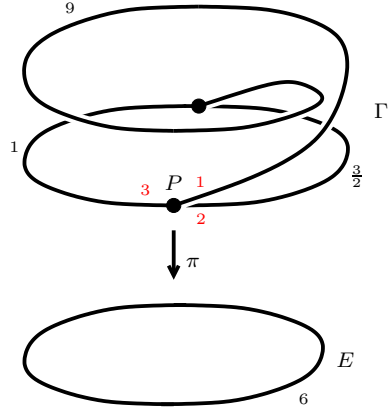}
\caption{A tropical Hurwitz cover of degree $4$ with genus $2$ source
curve.}\label{fig-tropcover}
\end{figure}
\end{example}

The multiplicity of a tropical Hurwitz cover of $E$ is analogous to Equation (\ref{eq-mult}), except balanced forks do not exist, all edges are bounded, and balanced wieners can be ''curled'' around the circle multiple times. 

\begin{definition}[Tropical simple Hurwitz number of $E$]\label{def-tropcovE}
Fix $g\geq 2$. Let $E$ be a tropical elliptic curve, i.e.\ a circle, with $2g-2$ vertices.
 The \emph{ tropical simple Hurwitz number} of $E$ $N_{d,g}^{\trop}$ is the weighted count of tropical Hurwitz covers of $E$ of degree $d$ with a source of genus $g$, where each cover is counted with its multiplicity.
 \end{definition}

\begin{example}
In Figure \ref{fig-ndgex}, we let $g=2$ and $E$ be a circle with two vertices. We depict tropical Hurwitz covers of degree $4$ and genus $2$ of $E$, where each picture has to be counted twice because it can be reflected to give another cover. The multiplicities for the top row are $4\cdot 2\cdot 2\cdot \frac{1}{2}=8$ (where the factor of $\frac{1}{2}$ arises due to the wiener of weight $2$), $4\cdot 3\cdot 1=12$, $3\cdot 2\cdot 1=6$. The multiplicities for the bottom row are $2\cdot 1\cdot 1=2$ and $2\cdot 1\cdot 1\cdot \frac{1}{2}$, where now the factor of $\frac{1}{2}$ arises due to the curled wiener of weight $1$. Altogether, we obtain $N_{d,g}^{\trop}= 2\cdot (8+12+6+2+1)=2\cdot 29=58$.

\begin{figure}
 \input{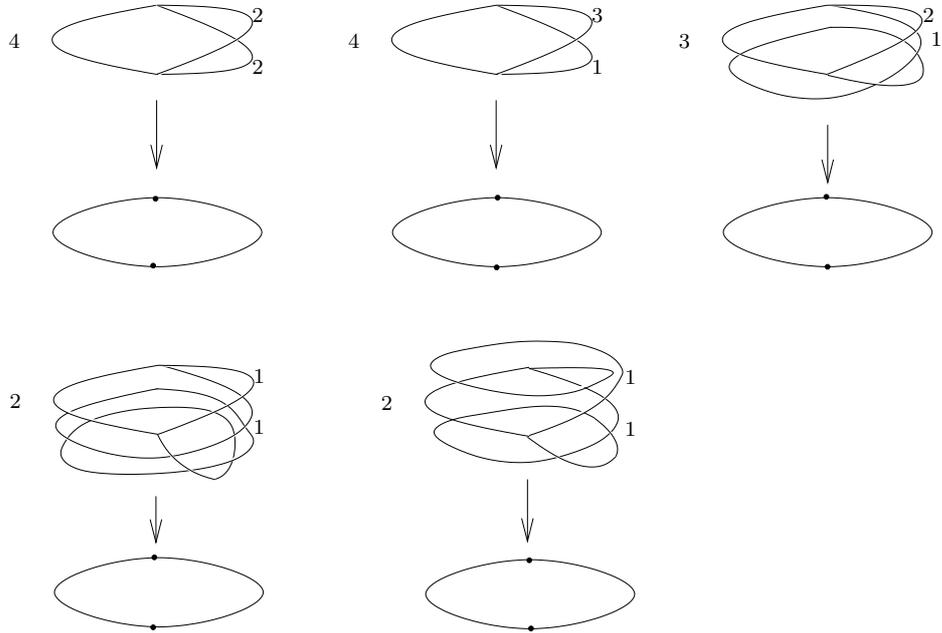}
\caption{The count of genus $2$ degree $4$ tropical Hurwitz covers of $E$.}\label{fig-ndgex}
\end{figure}
\end{example}

\begin{theorem} [The correspondence theorem for simple Hurwitz numbers of an elliptic curve] \label{thm-corresE} Fix an integer $g\geq 2$ and a degree $d$. Then the simple Hurwitz number of an elliptic curve $N_{d,g}$ equals its tropical counterpart, i.e.\
$$ N_{d,g}=N_{d,g}^{\trop}.$$
\end{theorem}

For a proof, see \cite{BBBM13} or \cite{BBM10}.

\subsection{Mirror symmetry for elliptic curves}\label{subsec-mirror}
The count of simply ramified degree $d$ genus $g$ covers of an elliptic curve plays a role in mirror symmetry for elliptic curves. Roughly, mirror symmetry is a duality relation among manifolds motivated by string theory. Important invariants of the manifolds get interchanged, notably the Hodge numbers, which can be arranged in a diamond shape, that is then reflected for the mirror manifold. An elliptic curve can be viewed as the smallest interesting case for mirror symmetry. The mirror is itself an elliptic curve, where some structural properties are changed.
Here, we focus on a statement relating the Hurwitz numbers of an elliptic curve with certain Feynman integrals, which we define in \ref{def-int}. These objects only depend on the topology of the elliptic curve, for that reason, we do not have to know more details about mirror elliptic curves.
The statement relating Hurwitz numbers and Feynman integrals (see Theorem \ref{thm-mirror}) was known before tropical covers played a role \cite{Dij95}. However, the tropical approach yields a finer relation which led to interesting new results on quasimodularity \cite{GM16}. Furthermore, tropical mirror symmetry for elliptic curves provides evidence and further strategies for the famous Gross-Siebert program on mirror symmetry which aims at constructing new mirror pairs and providing an algebraic framework for SYZ-mirror symmetry \cite{GS06, GS07, SYZ}. The Gross-Siebert philosophy how tropical geometry can be exploited is illustrated in the following triangle:

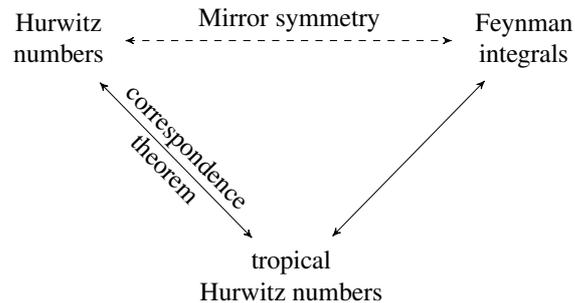
\[\scalebox{0.90}{
\begin{tikzpicture}[<->,>=stealth',shorten >=1pt,auto]
\coordinate (a) at (0,0);
\coordinate (b) at (-3.4,3.5);
\coordinate (c) at (3.4,3.5);
\node(1) at (a)  {\begin{tabular}{c}tropical\\ Hurwitz numbers\end{tabular}};
\node(2) at (b) {\begin{tabular}{c}Hurwitz\\ numbers\end{tabular}};
\node(3) at (c)  {\begin{tabular}{c}Feynman\\ integrals\end{tabular}};
\path[every node/.style={}]
(1) edge node[left,sloped, anchor=center] {\begin{tabular}{c}correspondence\\ theorem\end{tabular}} (2)
(2) edge [dashed] node {Mirror symmetry} (3)
(3) edge node[right,sloped, anchor=center] {} (1);
\end{tikzpicture}
}
\]

For manifolds other than an elliptic curve, the invariants have to be replaced by suitable generalizations. Naturally, we expect correspondence theorems to hold for the generalizations of Hurwitz numbers. The right arrow of the triangle is more mysterious, but the hope is that there is an equally natural connection relating tropical (generalized) Hurwitz numbers to (generalized) Feynman integrals. Then, the detour via tropical geometry can be used to prove relations like Theorem \ref{thm-mirror} among Hurwitz numbers and Feynman integrals for new pairs of mirror manifolds.
The fact that for elliptic curves, the right arrow that we call the tropical mirror symmetry relation (see Theorem \ref{thm-tropmirror}) naturally holds on a fine level supports this general strategy, or more generally, the idea to use tropical geometry as a tool in mirror symmetry.

We now define Feynman integrals, state the mirror symmetry theorems and give the idea how to use tropical geometry for the proof.

\begin{definition}[The propagator]\label{def-prop}
 We define the \emph{propagator} $$P(z,q):=\frac{1}{4\pi^2}\wp(z,q)+\frac{1}{12}E_2(q^2)$$ in terms of the Weierstra\ss{}-P-function $\wp$ and the Eisenstein series $$E_2(q):=1-24\sum_{d=1}^\infty \sigma(d)q^d.$$ Here, $\sigma=\sigma_1$ denotes the sum-of-divisors function $\sigma(d)=\sigma_1(d)=\sum_{m|d}m$.
\end{definition}
The variable $q$ above should be considered as a coordinate of the moduli space of elliptic curves, the variable $z$ is the complex coordinate of a fixed elliptic curve. (More precisely, $q=e^{i\pi \tau}$, where $\tau \in \mathbb{C}$ is the parameter in the upper half plane in the well-known definition of the Weierstra\ss{}-P-function.)

\begin{definition}[Feynman graphs and integrals]\label{def-int}
 A \emph{Feynman graph} $\Gamma$ of genus $g$ is a $3$-valent connected graph of genus $g$. For a Feynman graph, we throughout fix a reference labeling $x_1,\ldots,x_{2g-2}$ of the $2g-2$ vertices and a reference labeling $q_1,\ldots,q_{3g-3}$ of the edges of $\Gamma$.

For an edge $q_k$ of $\Gamma$ we call its neighbouring vertices $x_{k_1}$ and $x_{k_2}$ (the order plays no role).
Pick a total ordering $\Omega$ of the vertices and starting points of the form $iy_1,\ldots, iy_{2g-2}$ in the complex plane, where the $y_j$ are pairwise different small real numbers.
We define integration paths $\gamma_1,\ldots,\gamma_{2g-2}$ by $$\gamma_j:[0,1]\rightarrow \mathbb{C}:t\mapsto iy_j+t,$$ such that the order of the real coordinates $y_j$ of the starting points of the paths equals $\Omega$. 
We then define the integral \begin{equation}I_{\Gamma,\Omega}(q):= \int_{z_j\in \gamma_j} \prod_{k=1}^{3g-3} P(z_{k_1}-z_{k_2},q)dz_1\ldots dz_{2g-2}.\label{eq-Igamma}\end{equation}
\end{definition}

The integral thus depends on the combinatorics of the graph; the product in the integral can be viewed as a product over the edges. Since we integrate the variables $z_i$ away, we obtain a series in $q$.

\begin{example}
Figure \ref{fig-raupe} depicts a Feynman graph for $g=3$. For integration paths $\gamma_i$ chosen according to an order $\Omega$, the Feynman integral $I_{\Gamma,\Omega}(q)$ is 
$$ \int_{z_4\in \gamma_4} \cdots \int_{z_1\in \gamma_1}P(z_1-z_3,q)\cdot P(z_1-z_2,q)^2\cdot  P(z_2-z_4,q)\cdot P(z_3-z_4,q)^2 dz_1 dz_2 dz_3 dz_4.$$

\begin{figure}
\begin{center}
 \input{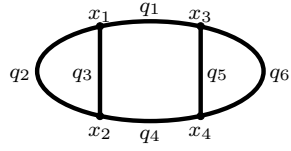}
\end{center}
\caption{A Feynman graph.}\label{fig-raupe}
\end{figure}
\end{example}

The following is the precise statement of the top arrow in the triangle above:
\begin{theorem}[Mirror Symmetry for elliptic curves]\label{thm-mirror}
Let $g>1$. For the definition of the invariants, see Section \ref{subsec-Hurwitz} and Definition \ref{def-int}. We have  $$ \sum_{d=1}^\infty N_{d,g} q^{2d}= \sum_{\Gamma} \frac{1}{|\Aut(\Gamma)|}\sum_\Omega I_{\Gamma,\Omega}(q),$$ where $\Aut(\Gamma)$ denotes the automorphism group of $\Gamma$, the first sum goes over all $3$-valent graphs $\Gamma$ of genus $g$ and the second over all orders $\Omega$ of the vertices of $\Gamma$.
\end{theorem}
The statement thus declares the equality of the generating series of Hurwitz numbers (by definition a series in $q$) with the sum of Feynman integrals, which is also a series in $q$ as we saw above. Despite the fact that both expressions are series in $q$, they have essentially nothing in common and the statement seems like magic. It is nice to see how tropical geometry can explain this magic by boiling everything down to a combinatorial hunt of monomials in generating functions which is bijectively related to tropical covers.

Using the correspondence Theorem \ref{thm-corresE}, it is obviously sufficient to prove the following theorem in order to obtain a tropical proof of Theorem \ref{thm-mirror}. This theorem can be viewed as the more mysterious right arrow in the triangle above.

\begin{theorem}[Tropical curves and integrals]\label{thm-tropmirror}
 Let $g>1$. For the definition of the invariants, see Definitions \ref{def-tropcovE} and \ref{def-int}. We have  $$ \sum_{d=1}^\infty N_{d,g}^{\trop} q^{2d}= \sum_{\Gamma}\frac{1}{|\Aut(\Gamma)|}\sum_{\Omega} I_{\Gamma,\Omega}(q),$$ where the first sum goes over all $3$-valent graphs $\Gamma$ of genus $g$ and the second over all orders $\Omega$ of the vertices of $\Gamma$.
\end{theorem}

To prove this theorem, it turns out that it is more natural to use finer structures keeping the labeling of the Feynman graph present.

For the tropical covers this means that we introduce \emph{labeled tropical covers} for which we require that the source is a metrization of a Feynman graph, with its labels. For an example, see Figure \ref{fig-labelledCover}. The edges
labeled $q_2,q_3$ and $q_6$ are supposed to have weight $1$, the edges $q_1$
and $q_4$ weight $2$ and $q_5$ weight $3$. The underlying Feynman graph is the one of Figure \ref{fig-raupe}.

What we gain is the possibility to define an \emph{order for a labeled tropical Hurwitz cover of $E$}: we fix a point $p_0$ in $E$ and we denote the vertices of $E$ starting at $p_0$ clockwise by $p_1,\ldots,p_{2g-2}$. We then define the order $\Omega$ for the cover by the order of the vertices $x_i$ which are mapped to $p_1,\ldots,p_{2g-2}$. For the example in Figure \ref{fig-labelledCover}, the order is $x_1<x_3<x_4<x_2$.

Furthermore, we can define a \emph{multidegree} for a labeled tropical Hurwitz cover of $E$ as the vector $\underline{a}=(a_1,\ldots a_{3g-3})$ that measures the degree in each edge: $$a_i= \#(\pi^{-1}(p_0)\cap q_i)\cdot w_i,$$
where $w_i$ denotes the weight of the edge $q_i$. For the example in Figure \ref{fig-labelledCover}, the multidegree is 
 $\underline{a}=(0,1,1,0,0,2)$. Notice that $q_6$, although it has weight $1$, contributes $2$ to the multidegree, because it is curled such that it passes over $p_0$ twice.

 \begin{definition}[Counts of labeled tropical covers] For a fixed Feynman graph $\Gamma$, a fixed multidegree $\underline{a}$ and an order $\Omega$, we let $N_{\underline{a},\Omega}^\Gamma$ be the number of labeled tropical Hurwitz covers of $E$ whose source equals $\Gamma$ and with the multidegree $\underline{a}$ and the order $\Omega$ (counted with multiplicity).
 \end{definition}
 
\begin{figure}
 \begin{center}
  \input{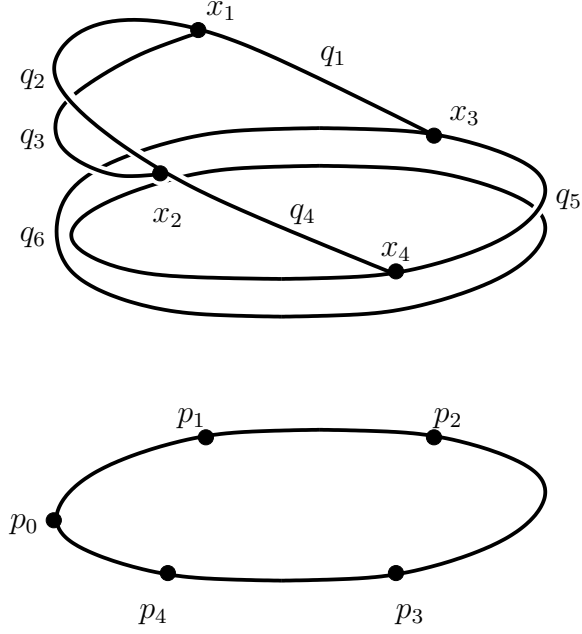}
 \end{center}
\caption{A labeled tropical cover of $E$.}\label{fig-labelledCover}
\end{figure}

For the Feynman integrals, we do not insert $q$ as second variable in the propagators but the edge variable $q_k$. For the graph in Figure \ref{fig-raupe}, we thus obtain 
\begin{align*} \int_{z_4\in \gamma_4} \cdots \int_{z_1\in \gamma_1}&P(z_1-z_3,q_1)\cdot P(z_1-z_2,q_2)\cdot P(z_1-z_2,q_3)\cdot \\&  P(z_2-z_4,q_4)\cdot P(z_3-z_4,q_5)\cdot P(z_3-z_4,q_6) dz_1 dz_2 dz_3 dz_4.\end{align*}

A \emph{refined Feynman integral} like this is a series in the $3g-3$ variables $q_1,\ldots,q_{3g-3}$. 
The tropical mirror symmetry theorem \ref{thm-tropmirror} follows from the following statement (after setting back $q_k=q$ and taking graph automorphisms into account \cite{BBBM13}):

\begin{theorem}[Labeled tropical covers and refined Feynman integrals]
For a fixed Feynman graph $\Gamma$, multidegree $\underline{a}$ and order $\Omega$ the number of labeled tropical Hurwitz covers of $E$, $N_{\underline{a},\Omega}^\Gamma$, equals the $q_1^{a_1}\cdot \ldots \cdot q_{3g-3}^{a_{3g-3}}$-coefficient of the series $I_{\Gamma,\Omega}(q_1,\ldots,q_{3g-3})$.
\end{theorem}

To prove this theorem, we first change coordinates using  $x_j=e^{i\pi z_j}$. Using the series expansion of the Weierstra\ss-P-function, one can show (Theorem 2.22 in \cite{BBBM13}) that the propagator factors $P(z_i-z_j,q_k)$ then become
\begin{equation}\label{PTaylor}
 P\left(\frac{x_{k_1}}{x_{k_2}},q_k\right)= \frac{\left(\frac{x_{k_1}}{x_{k_2}}\right)^2}{\left(\left(\frac{x_{k_1}}{x_{k_2}}\right)^2-1\right)^2}+\sum_{a_k=1}^\infty \left(\sum_{w_k|a_k}w_k \left(\left(\frac{x_{k_1}}{x_{k_2}}\right)^{2w_k}+ \left(\frac{x_{k_2}}{x_{k_1}}\right)^{2w_k}\right)\right)q_{k}^{2a_k}.
\end{equation}
For the (in $q_k$) constant term, we can use a geometric series expansion which depends on the order $\Omega$: if $x_{k_1}<x_{k_2}$, we have
\begin{equation} \sum_{w_k=1}^\infty w_k \left( \frac{x_{k_1}}{x_{k_2}}\right)  ^{2w_k}.\label{eq-constantterm}\end{equation}

We can combine (\ref{PTaylor}) and (\ref{eq-constantterm}) to obtain the following expression for $P\big(\frac{x_{k_1}}{x_{k_2}},q_k\big)$ assuming that $x_{k_1}<x_{k_2}$ in the order $\Omega$:

\begin{align}\label{Pexpansion}
& P\left(\frac{x_{k_1}}{x_{k_2}},q_k\right)=\nonumber \\& \sum_{w_k=1}^\infty w_k \left( \frac{x_{k_1}}{x_{k_2}}\right)  ^{2w_k}+\sum_{a_k=1}^\infty \left(\sum_{w_k|a_k}w_k \left(\left(\frac{x_{k_1}}{x_{k_2}}\right)^{2w_k}+ \left(\frac{x_{k_2}}{x_{k_1}}\right)^{2w_k}\right)\right)q_{k}^{2a_k}.
\end{align}

Using the coordinate changes $x_j=e^{i\pi z_j}$ in the integral, we have to multiply with the derivative of the inverse functions which amounts to a factor of $\frac{1}{x_1\cdot \ldots \cdot x_{2g-2}}$. The integration paths are taken to circles around the origin, which carries the pole. The residue theorem thus tells us that the integral equals the residue, or, neglecting the factor $\frac{1}{x_1\cdot \ldots \cdot x_{2g-2}}$ from above, the in all $x_i$ constant term of the product 

\begin{equation}  \prod_{k=1}^{3g-3} P\left(\frac{x_{k_1}}{x_{k_2}},q_k\right).\label{eq-Feynman}\end{equation}

Thus the following statement holds:

\begin{proposition}[Expansion for refined Feynman integrals]\label{prop-Feynman}
Let $\Gamma$ be a Feynman graph and $\Omega$ an order. The (multivariate) Feynman integral $I_{\Gamma,\Omega}(q_1,\ldots,q_{3g-3})$ equals the in all $x_i$ constant term of the product (\ref{eq-Feynman}) using the expansion (\ref{Pexpansion}) for each factor.
\end{proposition}

We now explain the rest of the ideas of the proof by following an example. Although in general an example is not a proof, this example should be sufficient to convey the main ideas.
We stick to the Feynman graph from Figure \ref{fig-raupe}, and to the order $\Omega$ with $x_1<x_3<x_4<x_2$.
We let $\underline{a} =(0,1,1,0,0,2)$. 
We claim that
\begin{equation} 2\cdot \left(\frac{x_1}{x_3}\right)^4 \cdot 1 \cdot   \left(\frac{x_2}{x_1}\right)^2 \cdot 1 \cdot\left(\frac{x_2}{x_1}\right)^2 \cdot 2 \cdot\left(\frac{x_4}{x_2}\right)^4 \cdot 3 \cdot\left(\frac{x_3}{x_4}\right)^6 \cdot 1 \cdot\left(\frac{x_4}{x_3}\right)^2 \label{eq-expr}\end{equation}
is a contribution to the $q_1^{a_1}\ldots q_{6}^{a_6}$-term of the in all $x_i$ constant term of the product (\ref{eq-Feynman}) after expanding, i.e.\ to the Feynman integral by Proposition \ref{prop-Feynman}.
Here, each factor of the form 
$$w_k\cdot \left(\frac{x_{k_1}}{x_{k_2}}\right)^{2w_k} $$
stands for the coefficient of $q_k^{a_k}$ of a summand taken from the sum $P\left(\frac{x_{k_1}}{x_{k_2}},q_k\right)$ in the product (\ref{eq-Feynman}) above, where each factor is as in (\ref{Pexpansion}).

It is easy to see that the $x_i$ all cancel and indeed we have a term which is constant in all $x_i$.
Furthermore, if $a_k=0$ (i.e.\ we are dealing with the constant term in $q_k$), we must because of (\ref{eq-constantterm}) ensure that the fraction respects the order $\Omega$ in the sense that if we have $\frac{x_i}{x_j}$, then $x_i<x_j$. We have $a_k=0$ for $k=1,4,5$, and for each of these fractions, the order is respected since $x_1<x_3$, $x_4<x_2$ and $x_3<x_4$. 

For $a_k\neq 0$, the half exponent $w_k$ must be a divisor of $a_k$ because of (\ref{Pexpansion}). For $k=2,3,6$ we must thus confirm that $1|1$, $1|1$ and $1|2$ which is the case.
Thus the claim holds.

Contributions to the Feynman integral like (\ref{eq-expr}) are in bijection with labeled tropical Hurwitz covers of $E$, and we now describe how to construct the corresponding cover:
\begin{itemize}
\item We draw the vertex $x_i$ above the points $p_j$ according to the order $\Omega$.
\item For each term $w_k\cdot \left(\frac{x_{k_1}}{x_{k_2}}\right)^{2w_k}$ for which $a_k=0$, we draw an edge $q_k$ of weight $w_k$ connecting $x_{k_1}$ to $x_{k_2}$ directly, without passing over the point $p_0$.
\item For each term $w_k\cdot \left(\frac{x_{k_1}}{x_{k_2}}\right)^{2w_k}$ for which $a_k\neq 0$, we draw an edge $q_k$ of weight $w_k$ connecting $x_{k_1}$ to $x_{k_2}$, but now we curl and pass over $p_0$, namely $\frac{a_k}{w_k}$ times.
\end{itemize}
We claim that what we obtain is a tropical Hurwitz cover of $E$ contributing to $N_{\underline{a},\Omega}^\Gamma$.
Let us first check that it has the right source: it must have the Feynman graph $\Gamma$ as source, because we have built it in such a way that each edge $q_k$ connects the vertices $x_{k_1}$ and $x_{k_2}$. 
Furthermore, the order $\Omega$ is satisfied because we built it that way: we required that the $x_i$ lie over the $p_j$ as given by $\Omega$. Also, it has multidegree $\underline{a}$, since we ensured that each edge with $a_k=0$ does not pass over $p_0$ at all, and that each edge with $a_k\neq0$ and with weight $w_k$ passes $\frac{a_k}{w_k}$ times over $p_0$, so its degree over $p_0$ is $a_k$.
Finally, why is it a tropical cover, i.e.\ why is it balanced at each $x_i$? The question whether $x_i$ appears in the denominator or numerator depends on the direction with which the corresponding edge enters $x_i$. The balancing condition is thus equivalent to the requirement that the product has to be constant in all $x_i$.

For our example (\ref{eq-expr}), the tropical labeled Hurwitz cover of $E$ we obtain is precisely the one from Figure \ref{fig-labelledCover}. At the vertex $x_1$ for example, the three edges $q_1$, $q_2$ and $q_3$ are adjacent. The edge $q_1$ is on the right of $x_1$, which corresponds to the fact that $x_1$ appears in the numerator of the fraction for $q_1$. The edges $q_2$ and $q_3$ are to the left, corresponding to the fact that $x_1$ appears in the denominator for both edges. The balancing condition stating that $2=1+1$, where $2$ is the weight of $q_1$ and $1$ are the weights of $q_2$ and $q_3$ is equivalent to the fact that $\frac{x_1^{2\cdot 2}}{x_1^{2\cdot 1}\cdot x_1^{2\cdot 1}}$ cancels in (\ref{eq-expr}).

Finally, the contribution to the Feynman integral given by our choice of monomials (\ref{eq-expr}) is $\prod w_k=2\cdot 1\cdot 1\cdot 2\cdot 3\cdot 1$, which equals the multiplicity with which the corresponding tropical cover is counted.

Thus we can see that the factors of the Feynman integral product (\ref{eq-Feynman}) represent the possibilities what an edge $q_k$ in a labeled tropical Hurwitz cover of $E$ can do: it can either not pass the point $p_0$ (which happens if and only if $a_k=0$) in which case it can have any possible weight $w_k$, but it must directly go from the smaller to the bigger vertex (see (\ref{eq-constantterm})). The sum of all these possibilities yields the in $q_k$ constant term
$\sum_{w_k=1}^\infty w_k \left( \frac{x_{k_1}}{x_{k_2}}\right)  ^{2w_k}$ (see (\ref{Pexpansion})).
Or it passes the point $p_0$ (which happens if and only if $a_k\neq 0$, in which case the order of the vertices plays no role, but the weight $w_k$ must be a divisor of $a_k$ (see (\ref{Pexpansion})), so we obtain 
$$\sum_{w_k|a_k}w_k \left(\left(\frac{x_{k_1}}{x_{k_2}}\right)^{2w_k}+ \left(\frac{x_{k_2}}{x_{k_1}}\right)^{2w_k}\right).$$

 So the magic explaining the right arrow of the mirror symmetry triangle is the fact that the propagator, build from the Weierstra\ss{}-P-function and the Eisenstein series, after the coordinate change, knows everything an edge in a labeled tropical cover can do.

\section{The cohomology of $\mathcal{M}_g$}\label{sec-Hom}
In a recent preprint by Chan, Galatius and Payne, the moduli space of tropical curves of genus $g$, $M^{\trop}_g$, is used to prove a surprising non-vanishing result about the cohomology of the moduli space of curves $\mathcal{M}_g$ for $g\geq 2$ \cite{CGP18}:
\begin{theorem}[Non-vanishing cohomology of $\mathcal{M}_g$ \cite{CGP18}] \label{thm-cohom}
The cohomology $H^{4g-6}(\mathcal{M}_g,\Q)$ is nonzero.
\end{theorem}

Here, $\mathcal{M}_g$ denotes the moduli space parametrizing all smooth algebraic curves of genus $g$. Alternatively, we can think about it as the space parametrizing all Riemann surfaces of genus $g$. Studying this space amounts to studying all smooth curves of genus $g$ simultaneously and is therefore highly rewarding. The fact that this space is a complex manifold of dimension $3g-3$ essentially goes back to Riemann. The cohomology of $\mathcal{M}_g$ has been a subject of intense study for decades (see e.g.\ \cite{Kir02,FPZ00,Tom05}). Readers not familiar with cohomology of algebraic varieties may, for the purpose of this text, just take for granted that $H^{i}(\mathcal{M}_g,\Q)$, $i \in \NN$, is an interesting object to care about.

The following conjecture on the cohomology of $\mathcal{M}_g$ was made by Church, Farb and Putman:
\begin{conjecture}[\cite{CFP14}] \label{con-CFP}
For each fixed $k$,  $H^{4g-4-k}(\mathcal{M}_g,\Q)$ vanishes for all but finitely many $g$.
\end{conjecture}
 While this holds true for $k=1$ \cite{CFP12, MSS13}, Theorem \ref{thm-cohom} shows that it is wrong for $k=2$.  Morita, Sakasai and Suzuki \cite{MSS15} pointed out that Conjecture \ref{con-CFP} is implied by a more general conjecture made by Kontsevich around 25 years ago \cite{Kon94a}. Thus, Theorem \ref{thm-cohom} also disproves Kontsevich's conjecture.

As we outline in the following two subsections, in the proof of Theorem \ref{thm-cohom}, the moduli space of tropical curves $M^{\trop}_g$ acts as an intermediary between $\mathcal{M}_g$ on the one hand and Kontsevich's graph complex $G^{(g)}$ on the other hand. We define the latter in Subsection \ref{subsecKont}. 

Then, the non-vanishing of the homology groups $H_0(G^{(g)})$ can be used to deduce non-vanishing results for the cohomology of $\mathcal{M}_g$.

\subsection{The homology of $M^{\trop}_g$}\label{subsec-homology}

We denote by $\tilde{H}_{2g-1}(M^{\trop}_g,\Q)$ the reduced rational homology of the simplicial complex obtained from $M^{\trop}_g$ by restricting to tropical curves for which the sum of all lengths of edges equals $1$ (see e.g.\ the red simplicial complex in Figure \ref{fig-m12glued}).
Recall that a \emph{simplicial complex} is obtained by gluing simplices along their boundaries. The associated \emph{chain complex} has vector spaces generated by the simplices in each dimension, and as differentials linear maps recording how the $d$-dimensional simplices are glued onto the $d-1$-dimensional ones. We form homology groups by taking the kernels of the differential $\partial_i$ modulo the image of $\partial_{i+1}$. Elements in the kernel are called \emph{cycles} and elements in the image \emph{boundaries}. In \cite{CGP18}, the concept is generalized to allow self-gluing of simplices like we can have them for $M^{\trop}_g$ (see Example \ref{ex-mgtrop}). It can be shown that the homology obtained like this equals the singular homology of the underlying topological space.

In the following theorem, the moduli space of tropical curves $M^{\trop}_g$ acts as intermediary between $\mathcal{M}_g$ on the one hand and Kontsevich's graph complex $G^{(g)}$ (to be defined, with its homology, in Subsection \ref{subsecKont}) on the other:

\begin{theorem}[Cohomology of $\mathcal{M}_g$, homology of $M_g^{\trop}$ and Kontsevich's graph complex] \label{ThmCGP}
{\color{white} hm}

\vspace{-0.5cm}

\begin{enumerate}
\item \label{ThmCGP1} There is a surjection from the cohomology group $H^{4g-6}(\mathcal{M}_g,\Q)$ to the reduced rational homology $\tilde{H}_{2g-1}(M^{\trop}_g,\Q)$ (see Theorem 1.2 \cite{CGP18}).
\item \label{ThmCGP2} There is an isomorphism $\tilde{H}_{2g-1}(M^{\trop}_g,\Q)\sim H_0(G^{(g)})$ (see Theorem 1.3 \cite{CGP18}).
\end{enumerate}
\end{theorem}
(Note that for both results, \cite{CGP18} acutally proves more general versions than we state here.)

Obviously, we can combine the statements \ref{ThmCGP}(\ref{ThmCGP1}) and \ref{ThmCGP}(\ref{ThmCGP2}) to obtain a surjection 

$$H^{4g-6}(\mathcal{M}_g,\Q) \rightarrow H_0(G^{(g)}).$$

Since by Theorem \ref{ThmH0Kont} below, $H_0(G^{(g)})$ is nonzero, we finally deduce the surprising non-vanishing of $H^{4g-6}(\mathcal{M}_g,\Q)$ stated in Theorem \ref{thm-cohom}, thus disproving Conjecture \ref{con-CFP} and the more general conjecture by Kontsevich mentioned above.

\subsection{Kontsevich's graph complex} \label{subsecKont}
The graph complex is a chain complex whose homology is given as cycles modulo boundaries (see Subsection \ref{subsec-homology}).

\begin{definition}[Kontsevich's graph complex] Let $\Gamma$ be a graph without loops such that every vertex has valence at least $3$. Let $\Omega$ be a total order on its edges.

We let the tuples $[\Gamma,\Omega]$, where $\Gamma$ is a graph of genus $g$ with $n$ edges $q_1,\ldots,q_n$, generate the rational vector space $G^{(g)}_n$, where the generators satisfy the relations
$[\Gamma,\Omega]=\mbox{sgn}(\sigma)[\Gamma',\Omega']$, if there exists an isomorphism of graphs $\Gamma \cong \Gamma'$ under which the total orderings $\Omega$ and $\Omega'$ are related by the permutation $\sigma \in \mathbb{S}_n$. 

The differential maps $[\Gamma,\Omega]$ to
$$\partial [\Gamma,\Omega] = \sum_{i=1}^n [\Gamma/q_i,\Omega|_{\Gamma/q_i}],$$
where $\Gamma/q_i$ is the graph for which the edge $q_i$ is contracted and its adjacent vertices are identified, and $\Omega|_{\Gamma/q_i}$ is the induced ordering on the edges of $\Gamma/q_i$.

In this way, we obtain Kontsevich's graph complex $G^{(g)}$.
\end{definition}
Notice that sign and grading conventions vary in the literature.

\begin{example}\label{ex-parallel}
Let $\Gamma$ be the graph of genus $3$ depicted in Figure \ref{fig-raupe}, where the vertex labelings can now be ignored. Consider the orders $\Omega=(q_1<\ldots<q_6)$ and $\Omega'=(q_2<q_1<q_3<\ldots<q_6)$, then $[\Gamma,\Omega]=[\Gamma,\Omega']$ since the edges $q_1$ and $q_2$ are parallel and hence indistinguishable.
With the automorphism exchanging the parallel edges $q_1$ and $q_2$, we also obtain $[\Gamma,\Omega]=-[\Gamma,\Omega']$, since the permutation exchanging the first two entries of the order is odd.
Thus $[\Gamma,\Omega]=-[\Gamma,\Omega']=-[\Gamma,\Omega]$, so
$[\Gamma,\Omega]=0$ in the chain complex. The same argument holds for any graph with two parallel edges.
\end{example}

\begin{example}
Let $W_g$ be the “wheel graph” with $g$ $3$-valent vertices, one $g$-valent vertex,  and  $2g$ edges  arranged  in  a  wheel  shape  with $g$ spokes. The graph $W_5$ is depicted in Figure \ref{fig-W5}. Fix an ordering of its edges. By abuse of notation, we will not mention the order in the following, since none of the following depends on the exact choice.

\begin{figure}
\begin{center}
\input{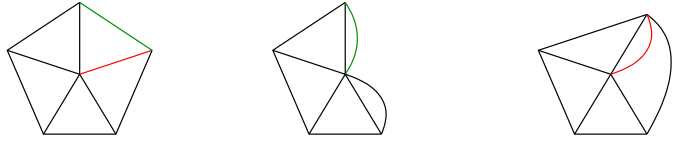}
\caption{The wheel graph $W_5$, the contraction of the red edge, and of the green edge.}\label{fig-W5}
\end{center}
\end{figure}
Any contraction of an edge $q$ in the wheel graph leads to a graph $W_g / q$ with two parallel edges and is thus zero when viewed as a chain in the graph complex. Hence $\partial [W_g]=0$. Thus, $[W_g]$ is a cycle in the homology of the graph complex.

The automorphism group of $W_g$ is the dihedral group containing rotations and reflections. 

If $g$ is even, a reflection can induce an odd permutation on the edges. E.g.\ if $g=4$, a diagonal reflection (i.e.\ a reflection fixing the inner and two nonadjacent outer vertices) induces a product of three transpositions of which two exchange two exterior edges and one two spokes. The remaining two spokes are fixed with the diagonal. 
Similar to Example \ref{ex-parallel} it then follows that $[W_g]=0$ as a chain.

If $g$ is odd, it can be shown that any rotation or reflection yields an even permutation of the edges. It follows that $[W_g]\neq 0 $ as a chain. This is only on the level of chains however, it is yet again a different story to prove that $[W_g]\neq 0$ when viewed as a cycle in the homology of the graph complex.
\end{example}

\begin{theorem}[The wheel graph is no boundary]\label{ThmH0Kont}
We have $[W_g]\neq 0$ in $H_0(G^{(g)})$, i.e.\ the cycle of the wheel graph is not a boundary. In particular, $H_0(G^{(g)})\neq 0$.
\end{theorem}
This follows by work of Willwacher \cite{Wil15} identifying the degree $0$ cohomology of the dual cochain complex $\prod_{g=2}^\infty \Hom(G^{(g)},\Q)$ with the Grothendieck-Teichm\"uller Lie algebra (see e.g.\ \cite{Sch97, Dri90}).

\end {document}